\documentclass{article}

\usepackage[utf8]{inputenc}
\usepackage{amsmath}
\usepackage{amssymb}
\usepackage{amsthm}
\usepackage{mathrsfs}
\usepackage{float}
\usepackage{url}

\usepackage{neurips}

\usepackage{graphicx}
\graphicspath{ {./images/} }

\theoremstyle{definition}

\newtheorem{prop}{Proposition}
\newtheorem*{theorem*}{Theorem}
\newtheorem*{theorempf*}{{\bfseries The Perron-Frobenius Theorem} \citep{perron_1907}}
\newtheorem{lemm}[prop]{Lemma}
\newtheorem{conj}[prop]{Conjecture}
\numberwithin{prop}{subsection}

\title{Tropical Analysis of the Asymptotics of the Perron-Frobenius Eigenvector}
\author{Balazs Kustar}
\date{May 2019}

\usepackage{natbib}
\usepackage{graphicx}

\begin{document}

\maketitle

\begin{abstract}
Asymptotic properties of matrices are, in general, difficult to analyze with classical mathematical techniques. In very specific cases, there is a well-known connection between the asymptotic behavior of a matrix's leading eigenvector and the corresponding "tropical" matrix, arising out of the $max$ and $min$ operations innate in tropical analysis. In this paper we examine a more general class of matrices, and explore the extent to which we can generalize the results using tropical techniques. We find that while the original results do not easily generalize, we can still make some useful statements about the asymptotic behavior in the general case, and can give a complete characterization for a larger class of matrices than previously examined.
\end{abstract}

\section{Introduction}

In statistical physics, one is often interested in modeling the behavior of a 1-dimensional system. Such a system can be described using its partition function \citep{santra}, which is written in terms of a matrix called its \textit{transfer matrix} with a real parameter $h$, the temperature of the system. This matrix has the form
$$(A_h)_{ij} = e^{(h^{-1}A_{ij})}$$
for some matrix A, of which the real parameters $A_{ij}$ represent potential terms plus interaction energy terms (when two adjacent sites are in states $i$ and $j$ respectively) \citep{reischuk}. The partition function depends on the trace of this matrix, and since the trace of a matrix is the sum of its eigenvalues, the spectrum of the transfer matrix can reveal important information about the system. A particularly interesting case is when the temperature $h$ goes to zero. Then the Perron Eigenvalue $\rho(A_h)$ is used to determine the free energy per site $\delta_h = h \log(\rho(A_h))$.

Spectral theory of tropical matrices plays an important role in the analysis of discrete event systems \citep{baccelli}. In fact, the asymptotic behavior of $\rho(A_h)$ (and thus the free energy of the system) can be expressed elegantly using tropical algebra under certain conditions. It is a known result that as $h\rightarrow0$, $h \log(\rho(A_h))$ is equal to the tropical max-plus eigenvalue of $A$, and when $A$ has only one critical tropical eigenvector, the eigenvector corresponding to $\rho(A_h)$ is equal to the tropical eigenvector of $A$ \citep{PhysMotiv}. In this paper we explore the behavior of that corresponding eigenvector in the general case and seek to answer whether or not a tropical connection still exists.

We first explore whether or not the convergence of the normalized Perron Eigenvector of $A_h$ is always to a critical tropical eigenvector of the base matrix $A$. Experiments quickly show that it is not, but it does always lie in the tropical eigenspace. We find that (Conjecture 3.0.1) the convengence appears to be determined by the shape of the tropical eigenspace (which is in turn determined by the set of critical tropical eigenvectors of A). As the tropical eigenspace is often determined by relatively few entries of A, we can perturb several entries of our base matrix without affecting the limit of the Perron Eigenvector. We also find that (Conjecture 3.0.2) when all tropical eigenvectors have form $(0, v_1+\alpha, v_2+\alpha, \dots, v_n+\alpha) \in \mathbb{TP}^{n-1}$ for $0 \leq \alpha \leq \beta \in \mathbb{R}$, we can give an exact characterization of the convergence.

We also attempt to apply the methods proposed in \citep{PertTheory}, as our matrices are of the form studied therein. However, we find that when $A$ has several critical classes, the algorithm in \citep{PertTheory} often fails to predict the Perron Eigenvector, giving instead an asymptotic characterization of another eigenvector, or none at all.

\section{Background}

\subsection{Tropical Algebra}

For our analysis we introduce the three \textit{tropical algebras}. The \textit{max-plus tropical semiring} (max+) is $\mathbb{R} \cup \{-\infty\}$ equipped with two operations $\oplus$ and $\odot$ defined as:
$$ a \oplus b = \max(a, b) $$
$$ a \odot b = a+b $$

This tropical semiring is associative and distributive, with additive identity $-\infty$ and multiplicative identity 0. This satisfies all ring axioms except for the existence of additive inverses, and so is a semiring. We will use $\mathbb{R}$ to refer to the base set of the tropical semirings for ease of notation.

The n-dimensional real vector space $\mathbb{R}^n$ is a module over the tropical semiring $(\mathbb{R}, \oplus, \odot)$, with the operations of coordinate-wise tropical addition:
$$(a_1, \dots , a_n) \oplus (b_1, \dots , b_n) = (max(a_1, b_1), \dots , max(a_n, b_n))$$
and tropical scalar multiplication:
$$\lambda \odot (a_1, \dots , a_n) = (a_1 + \lambda, \dots , a_n + \lambda)$$

We can define tropical matrix operations, exponents, and polynomials, etc. by replacing the classical addition and multiplication operations with the tropical analogs. Tropical $n \times n$ matrices have unique eigenvalues, with an associated eigenspace formed by the tropical convex hull of up to $n$ critical eigenvectors.

A tropical linear space $L$ in $\mathbb{R}^n$ consists of all tropical linear combinations of a fixed finite subset $\{a, b, \dots , c\} \subset \mathbb{R}^n$.
$$ L := span_{trop}(a,b,...,c) = \left\{ \lambda_a \odot a \oplus \lambda_b \odot b \oplus \dots \oplus \lambda_c \odot c : \lambda_a, \lambda_b, \dots \lambda_c \in \mathbb{R} \right\} $$
Note that $L$ is closed under tropical scalar multiplication: $ L = L + \mathbb{R}(1,1,\dots 1) $. We therefore choose to identify $L$ (and often, individual tropical vectors in $\mathbb{R}^n$) with its image in the \textit{tropical projective space}:
$$ \mathbb{TP}^{n-1} = \mathbb{R}^n / \mathbb{R}(1,1, ... ,1) $$

The \textit{min-plus tropical semiring} $(min+)$ is defined similarly, but with $ a \oplus b = \min(a, b) $ (and additive identity $\infty$), while the \textit{max-times tropical semiring} $(max*)$ defines $ a \odot b = ab$.

We will let $\lambda_{max+}(M)$ denote the max-plus eigenvalue of a matrix $M$, and $\sigma_{max+}(M)$ the corresponding eigenspace. Similarly, we define $\lambda_{max*}(M)$, $\sigma_{max*}(M)$, $\lambda_{min+}(M)$, and $\sigma_{min+}(M)$

All three tropical semirings satisfy the same ring axioms, and statements in one algebra have corresponding statements in the others. One can move between (min+) and (max+) by negating all values, or between (max+) and (max*) by exponentiation/logarithms.

\newpage

Example:
\[A \odot_{(max*)} x = \lambda \odot_{(max*)} x \]
\[\iff\]
\[log(A) \odot_{(max+)} log(x) = log(\lambda) \odot_{(max+)} log(x)\]

When $A$ is a matrix and $x$ a vector we use element-wise exponentiation or logarithms to achieve the equivalence.

\subsection{Visualizing Tropical Vectors}

We can use projections to $\mathbb{TP}^{n-1}$ to visualize tropical vectors using Euclidean space of a smaller dimension.
For a vector $v = (v_1, v_2 \dots , v_n) \in \mathbb{TP}^{n-1}_{max+}$,
$$ v = (v_1)^{\odot -1} \odot v = (0, v_2-v_1, \dots , v_n-v_1)$$
We use this fact to "normalize" tropical vectors so that the first coordinate is the tropical multiplicative identity, and project to $\mathbb{R}^{n-1}$ using the remaining coordinates.

\[
    v = \begin{bmatrix}
           v_1 \\
           v_2 \\
           \vdots \\
           v_n
         \end{bmatrix} \mapsto \begin{bmatrix}
           0 \\
           v_2-v_1 \\
           \vdots \\
           v_n-v_1
         \end{bmatrix} \mapsto \begin{bmatrix}
           v_2-v_1 \\
           \vdots \\
           v_n-v_1
         \end{bmatrix} \in \mathbb{R}^{n-1}
\]

For $n=3$ we have a projection to the Euclidean plane $\mathbb{R}^2$. This allows for easy visualization of the $n=3$ case.

\begin{figure}[H]
\includegraphics[width=7cm]{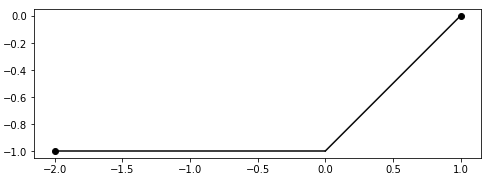}
\centering
\caption[Visualization of the tropical span]{
    Visualization of the tropical span of 
    $v_1 = \left[ \begin{smallmatrix} 0 \\ -2\\-1 \end{smallmatrix} \right]$
    and $v_2 = \left[ \begin{smallmatrix} 0 \\ 1\\0 \end{smallmatrix} \right]$
}
\centering
\end{figure}

\subsection{The Perron Eigenvalue}

\begin{theorempf*}

Let $A\in \mathbb{R}^{n \times n}$ be a matrix with strictly positive entries. Then the following statements hold:

\begin{enumerate}
\item There is a positive real number $r$, such that $r$ is an eigenvalue of $A$ and for any other eigenvalue $\lambda$ of $A$, $|\lambda| < r$.
\item The eigenspace associated to $r$ is one-dimensional.
\item There exists an eigenvector $v$ of $A$ (with eigenvalue $r$) such that all entries of $v$ are positive.
\item There are no other non-negative eigenvectors of $A$, other than positive multiples of $v$.

\end{enumerate}

\end{theorempf*}

We call this maximal eigenvalue the \textit{Perron Eigenvalue} $\rho(A)$, and the corresponding positive eigenvector the \textit{Perron Eigenvector} $\mathscr{L}(A)$.

\medskip

Let $M \in \mathbb{R}^{n \times n}$ be a matrix such that $M_{ij}> 0$ for all $i, j$. By the Perron-Frobenius theorem (P-F), $\exists \lambda = \rho(M), x \in \mathbb{R}^n$ such that $Mx = \lambda x$, $\lambda>0$, and $x_i>0$ $\forall i \in [n]$.

Let $M^{(k)}$ be the $k$-th Hadamard power of M. $(M^{(k)})_{ij} = (M_{ij})^k$. Since all $M_{ij} > 0$, we know that $M^{(k)} > 0$ and so the P-F theorem applies to it as well.

Let $\lambda_k$ be the Perron Eigenvalue $\rho(M^{(k)})$.
Then by P-F:
$$\exists x_k = \mathscr{L}(M^{(k)}) \text{ s.t. }\  (M^{(k)}x_k)_i = \lambda_k (x_k)_i \quad \forall i$$

For large $k$, the largest term in each row of M will dominate:
$$(M^{(k)}x_k)_i = \sum_j (M_{ij})^k*(x_k)_j \ \approx \max_j \left\{ (M_{ij})^k*(x_k)_j \right\}$$

So as $k \rightarrow \infty$:
$$\lambda_k (x_k)_i = (M^{(k)}x_k)_i = \max_j \left\{ (M_{ij})^k*(x_k)_j \right\}$$
$$ = (M^{(k)})_i \odot_{(max*)} x_k $$

This must hold for all $i$, so combining, we have:
$$ \lambda_k x_k = M^{(k)} \odot x_k $$

And so as $k \rightarrow \infty$, $\lambda_k$ must be a tropical (max-times) eigenvalue of $M^{(k)}$, with $x_k$ a corresponding tropical eigenvector.

\medskip

\begin{lemm}
$\lambda_{max*}(M^{(k)})=\lambda_{max*}(M)^k$
\end{lemm}

\begin{lemm}
$v \in \sigma_{max*}(M) \iff v^{(k)} \in \sigma_{max*}(M^{(k)})$
\end{lemm}

\begin{proof}
Since M is a positive matrix, for any $a,b,c,d \in \{M_{ij} : i,j \in [n]\}$
$$ ab \leq cd \iff a^kb^k \leq c^kd^k$$
This means that any cycle that achieved the maximum mean cost in $M$, must also achieve it in $M^{(k)}$.
Let $\sigma = (\sigma_1, \sigma_2, \dots , \sigma_s)$ be a cycle that achieves the maximum mean cost. Then:
$$ \lambda_{max*}(M)^k = \left( \prod_{i=2}^{s} M_{\sigma_{i-1}\sigma_i}\right)^{\frac{1}{s}k} = \left( \prod_{i=2}^{s} (M_{\sigma_{i-1}\sigma_i})^k\right)^{\frac{1}{s}} = \lambda_{max*}(M^{(k)})$$

Let $M^* = M \oplus M^{\odot 2} \oplus  M^{\odot 3} ... \oplus M^{\odot n}$ be the Kleene star of $M$, so that the tropical critical eigenvectors of M can be read off of the columns of $M^*$

A similar argument shows that $(M^*)^{(k)} = (M^{(k)})^*$, and therefore the eigenvectors of $M^{(k)}$ are exactly the $k$'th hadamard powers of eigenvectors of $M$.

\end{proof}

We can then conclude that as $k \rightarrow \infty$:
$$ \lambda_k^{1/k} = \lambda_{max*}(M) $$
$$ x_k^{(1/k)} \in  \sigma_{max*}(M)$$

Thus when $\sigma_{max}(M)$ has only a single basis element, it must be equal to $\lim_{k \rightarrow \infty}\left(x_k^{(1/k)}\right)$ (in projective space).

We now examine the asymptotics of the transfer matrices defined above. For convenience we will consider instead the matrices $A_k = exp(kA)$ as $k \rightarrow \infty$.

\bigskip

\begin{prop}
Let $A \in \mathbb{R}^{n \times n}$ be a matrix. The limit of the normalized Perron Eigenvalue of $A_k = exp(kA)$ as $k \rightarrow \infty$  is equal to the tropical max-plus eigenvalue of $A$.
\end{prop}

\begin{proof}$ $\newline
We are interested in describing $\lambda_{PF} = \lim_{k \rightarrow \infty}\left(\frac{1}{k} \log(\rho(A_k))\right)$. Let $M = exp(A)$ so that $A_k = M^{(k)}$. 

$(M)_{ij} = e^{A_{ij}} > 0 \implies M$ is a positive matrix and satisfies the conditions above.
$$ \lim_{k \rightarrow \infty}\left(\frac{1}{k} \log(\rho(A_k)) \right)= \lim_{k \rightarrow \infty}\left(\log(\rho(M^{(k)})^{\frac{1}{k}}) \right) $$
$$ = \log(\lambda_{max*}(M))$$
\smallskip
$$\text{Recall that by taking the logarithm we can convert from max-times to max-plus algebra:}$$
\smallskip
$$ = \lambda_{max+}(\log(M))$$
$$ = \lambda_{max+}(\log(exp(A)))$$
$$ = \lambda_{max+}(A)$$
\end{proof}

\subsection{The Perron Eigenvector}

While the Perron Eigenvalue of transfer matrices can be described quite simply using tropical mathematics, The Perron Eigenvector is less straightforward. We first examine the normalized Perron Eigenvector $\mathscr{P}_k = \frac{1}{k}\log(\mathscr{L}(A_k))$ using the same tools from above.
\newline

\begin{prop}
$\frac{1}{k}\log(\mathscr{L}(A_k)) \in \sigma_{max+}(A)$ as $k \rightarrow \infty$
\end{prop}
\begin{proof}
$$(\mathscr{L}(M^{(k)}))^{(1/k)} \in \sigma_{max*}(M) \implies \log \left( (\mathscr{L}(M^{(k)}))^{(1/k)} \right) \in \sigma_{max+}(\log(M)) $$
$$ \implies \frac{1}{k}\left( \log(\mathscr{L}(M^{(k)})) \right) \in \sigma_{max+}(A) $$
$$ \implies \frac{1}{k}\log(\mathscr{L}(A_k)) \in \sigma_{max+}(A) $$
\end{proof}
When $\sigma_{max+}(A)$ has only one generating eigenvector, we can fully characterize the asymptotics of the eigenvector (up to projection to $\mathbb{TP}^{n-1}$). But that is a strong requirement on $A$, and one that is relatively difficult to define classically.

\subsection{Perturbation Theory}
In \citep{PertTheory} the authors give an algorithm for describing the first order asymptotics of eigenvectors of matrices of the form $\mathcal{A}_\epsilon$ such that
$$ (\mathcal{A}_\epsilon)_{ij} = a_{ij}\epsilon^{B_{ij}} $$
as the real parameter $\epsilon > 0$ tends to 0. This relies on a decomposition of the base matrix B using $(min+)$ algebra. A sequence of matrices is generated using the repeated Schur complement of the tropical critical classes, the nodes that lay on a cycle that achieves the eigenvalue as the mean cycle length.

Let $C \subset L$ be finite sets, and let $N = L \setminus C$. If $A$ is an $L \times L$
matrix with entries in $\mathbb{R}$, the min-plus Schur complement of $C$ in $A$ is defined as:
$$Schur(C, A) = A_{NN} \oplus A_{NC} (A_{CC})^* A_{CN} $$

Then the collection of eigenvalues of each iteration of the Schur complement is used to normalize the base matrix $B$ to obtain $\hat{B}$. Then $\hat{B}^*$ is used to predict the asymptotic behavior of eigenvalues of $\mathcal{A}_\epsilon$ in the form of a weight vector $w$ and an exponent vector $v$ so that a predicted eigenvalue $V_\epsilon$ has asymptotic behavior (in terms of $\epsilon$):
$$ V_\epsilon = (w_1\epsilon^{v_1}, w_2\epsilon^{v_2}, \dots, w_n\epsilon^{v_n})$$
For a full description of the algorithm, see sections 5-6 of \citep{PertTheory}.

The prediction relies on a choice of a level of Schur decomposition $l$, as well as a choice of (classical) eigenvalue of the corresponding matrix and a choice of vector from the columns of $\hat{B}^*$. Of these choices, many fail to give any characterization of an eigenvector, as when the output vector $w$ has $w_i=0$, Theorem 6.1 states that:

$$\frac{(V_\epsilon)_i}{(V_\epsilon)_j} \propto 0\epsilon^{v_i-v_j}$$

Which does not give a full description of the vector. When $w$ is non-zero for all entries, $V_\epsilon$ does predict an eigenvector of $A_\epsilon$ up to a multiplicative constant.

\section{Experiments}

We focus our experiments on the general case where $A$ has several critical classes. For visualization, we narrow our focus to the $n=3$ case, as it is easy to plot points projected to $\mathbb{TP}^2$ using the Euclidean plane.

We first seek to determine whether the convergence of the normalized P-F eigenvector is simply to one of the critical eigenvectors of A.

\begin{figure}[H]
\minipage{0.32\textwidth}
  \includegraphics[width=\linewidth]{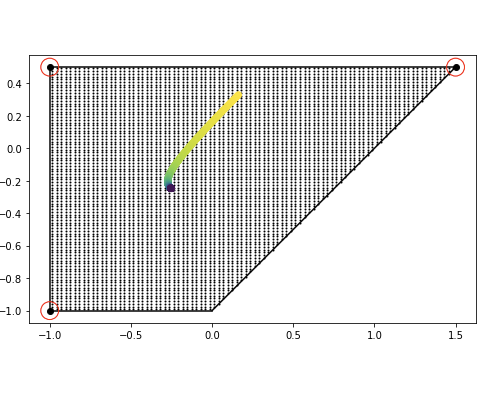}
  \caption[]{$A = \left[ \begin{smallmatrix} 0 & -2.5 & -0.5\\ -1 & 0 & -1.5 \\ -1 & -1 & 0\end{smallmatrix}\right]\newline \hphantom{xxxx} \mathscr{P}_{\infty}(A) = \left[ \begin{smallmatrix} 0 \\ -0.25\\-0.25 \end{smallmatrix} \right]$}\label{figure1}
\endminipage\hfill
\minipage{0.32\textwidth}
  \includegraphics[width=\linewidth]{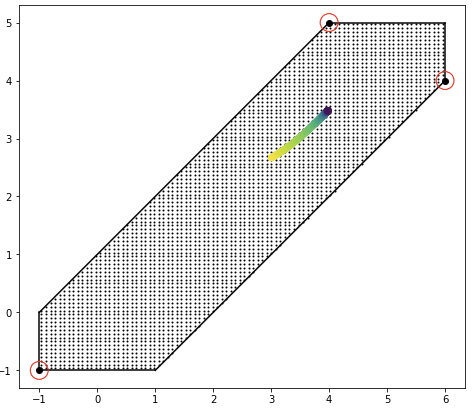}
  \caption[]{$A = \left[ \begin{smallmatrix} 0 & -6 & -5\\ -1 & 0 & -1 \\ -1 & -2 & 0\end{smallmatrix}\right]\newline \hphantom{xxxxx} \mathscr{P}_{\infty}(A) = \left[ \begin{smallmatrix} 0 \\ 4\\3.5 \end{smallmatrix} \right]$}\label{figure2}
\endminipage\hfill
\minipage{0.32\textwidth}%
  \includegraphics[width=\linewidth]{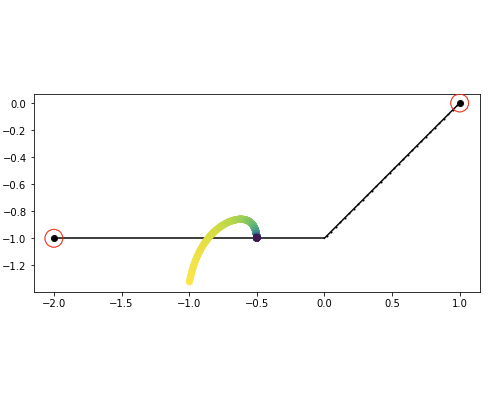}
  \caption[]{$A = \left[ \begin{smallmatrix} 0 & -1 & -1\\ -4 & 0 & -1 \\ -1 & -1 & -4\end{smallmatrix}\right]\newline \hphantom{xxxx} \mathscr{P}_{\infty}(A) = \left[ \begin{smallmatrix} 0 \\ -0.5\\-1 \end{smallmatrix} \right]$}\label{figure3}
\endminipage
\end{figure}

\begin{figure}[H]
\includegraphics[width=5cm]{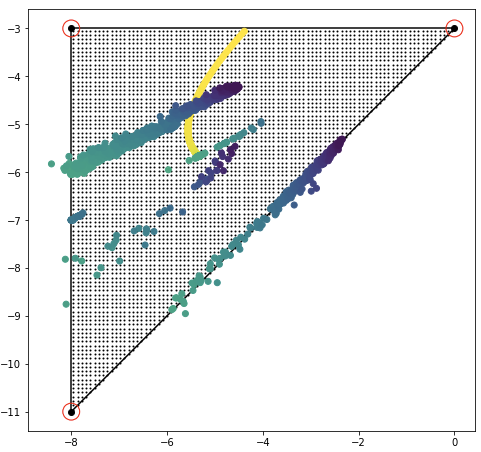}
\centering
\caption{Loss of machine precision creates undesired artifacts after convergence.}
\centering
\end{figure}

In the figures in this section, the shaded area is the $(max+)$ eigenspace of the base matrix A. The critical eigenvectors are circled in red. The normalized Perron Eigenvector for $A_k$ as is shown as a sequence of points colored yellow $\rightarrow$ purple as $k\rightarrow \infty$. The iteration is cut off around $k=30$ due to loss of machine precision causing undesired behavior (see figure 5).

An immediate observation from the plots are that the Perron Eigenvector does converge to a single point, with no subsequences even when there are multiple critical eigenvectors of the base matrix, so the limit
$$ \mathscr{P}_{\infty}(A) =  \lim_{k \rightarrow \infty} \left( \frac{1}{k}\log(\mathscr{L}(A_k)) \right) $$
is well-defined.

We can see that $\mathscr{P}_{\infty}(A) \in \sigma_{max+}(A)$ as predicted, however it does not tend to converge to one of the critical eigenvectors, except in the case of conjecture 1. Rather, the convergence point tends to be a (not necessarily a homogenous) linear combination of all critical eigenvectors.

\begin{figure}[H]
\minipage{0.49\textwidth}
  \includegraphics[width=\linewidth]{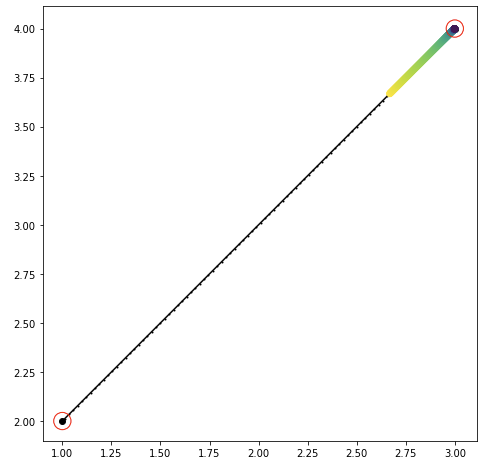}
  \caption[]{$A = \left[ \begin{smallmatrix} 0 & -3 & -2\\ 1 & 0 & -1 \\ 2 & 1 & 0\end{smallmatrix}\right] \mathscr{P}_{\infty}(A) = \left[ \begin{smallmatrix} 0 \\ 3\\4 \end{smallmatrix} \right]$}\label{figure4}
\endminipage\hfill
\minipage{0.49\textwidth}
  \includegraphics[width=\linewidth]{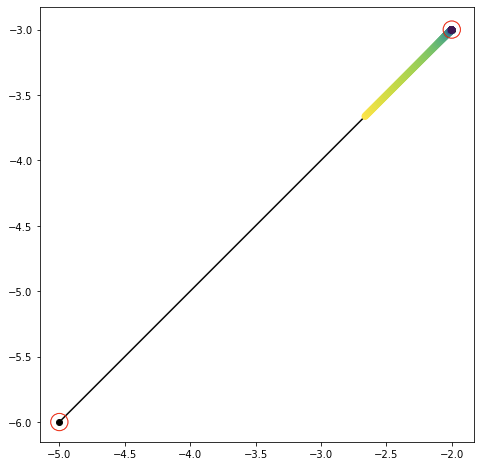}
  \caption[]{$A = \left[ \begin{smallmatrix} 0 & 1 & 3\\ -5 & 0 & 1 \\ -6 & -1 & 0\end{smallmatrix}\right] \mathscr{P}_{\infty}(A) = \left[ \begin{smallmatrix} 0 \\ -2\\-3 \end{smallmatrix} \right]$}\label{figure5}
\endminipage
\end{figure}

We now state two conjectures without proof that are supported by all experimentation. These conjectures both appear to generalize to $n \times n$ matrices, but the exploration becomes more difficult without visualization.

\begin{conj} 
Let $v = (0,v_1,v_2) \in \sigma_{max+}(A)$. If for any other $w = (0,w_1,w_2) \in \sigma_{max+}(A)$ there exists an $\alpha > 0 \in \mathbb{R}$ such that $v_1=w_1+\alpha$ and $v_2=w_2+\alpha$, then $v = \mathscr{P}_{\infty}(A)$
\end{conj}
When the eigenspace is a classical line segment in our projected space, $\mathscr{P}_{\infty}(A)$ lies on the critical eigenvector closest to $(0, \infty, \infty)$, as can be seen in figures 6-7. 

As a generalization, we claim that if all eigenvectors of A have the form $(0, v_1+\alpha, v_2+\alpha, \dots, v_n+\alpha)$ for $0 \leq \alpha \leq \beta$, then $\mathscr{P}_{\infty}(A) = (0, v_1+\beta, v_2+\beta, \dots, v_n+\beta)$

\begin{figure}[H]
\minipage{0.49\textwidth}
  \includegraphics[width=\linewidth]{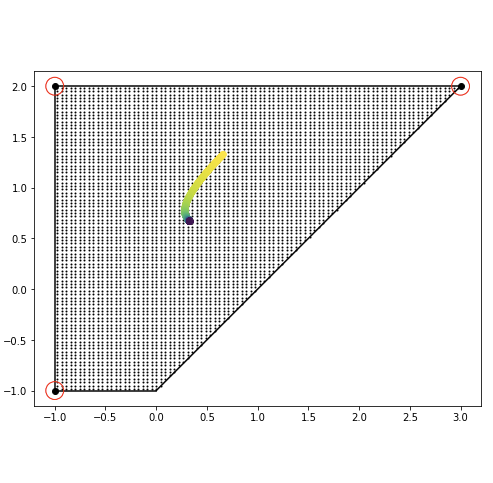}
  \caption[]{$A = \left[ \begin{smallmatrix} 0 & -4 & -2\\ 1 & 0 & -3 \\ -1 & -1 & 0\end{smallmatrix}\right] \mathscr{P}_{\infty}(A) = \left[ \begin{smallmatrix} 0 \\ 1/3\\2/3 \end{smallmatrix} \right]$}\label{figure6}
\endminipage\hfill
\minipage{0.49\textwidth}
  \includegraphics[width=\linewidth]{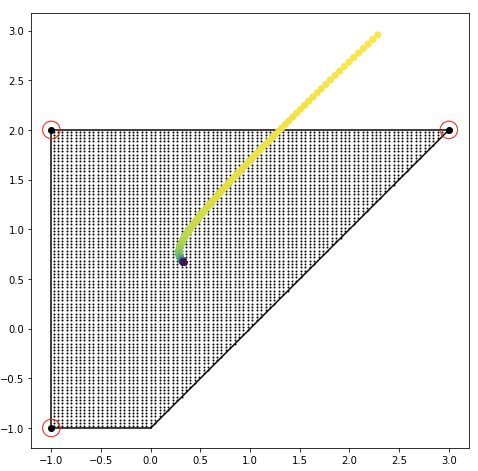}
  \caption[]{$A = \left[ \begin{smallmatrix} 0 & -9 & -2\\ 1 & 0 & -3 \\ -1 & -1 & 0\end{smallmatrix}\right] \mathscr{P}_{\infty}(A) = \left[ \begin{smallmatrix} 0 \\ 1/3\\2/3 \end{smallmatrix} \right]$}\label{figure7}
\endminipage
\end{figure}

\newpage

\begin{conj}
$\sigma_{(max+)}(A) = \sigma_{(max+)}(B) \implies \mathscr{P}_{\infty}(A) = \mathscr{P}_{\infty}(B)$
\end{conj}

As can be seen in figures 8-9, matrices that have the same tropical eigenspace have the same $\mathscr{P}_{\infty}$. This means that $\mathscr{P}_{\infty}$ is a function of the eigenspace. Specifically, if we let $E(A) \subset \mathbb{TP}^{n-1}$ be the set of critical eigenvectors of $A$, so that $conv(E(A)) = \sigma_{(max+)}(A)$, then
$$\mathscr{P}_{\infty}(A) = f\left( conv(E(A))) \right) = g(E(A))$$
for some function $g: (\mathbb{TP}^{n-1})^n \rightarrow \mathbb{TP}^{n-1}$.
Assuming that we only care about $\mathscr{P}_{\infty}(A)$ up to projection to $\mathbb{TP}^{n-1}$, learning this function would completely determine the asymptotics.

Note that the tropical eigenspace of $A$ is often determined by only a few entries of the matrix. Classically one would expect that perturbations anywhere in the matrix should change the asymptotic behavior, but because of the $max$ operations fundamental to tropical algebra, many of the entries do not affect the eigenspace and can therefore be perturbed without affecting the convergence.

Another way to analyze the asymtotics of $A_k$ is proposed in \citep{PertTheory}. They give an algorithm for describing the first order asymptotics of eigenvectors of matrices of the form $\mathcal{A}_\epsilon$ such that
$$ (\mathcal{A}_\epsilon)_{ij} = a_{ij}\epsilon^{B_{ij}} $$
as the real parameter $\epsilon > 0$ tends to 0.

Recall that $(A_k)_{ij} = e^{kA_{ij}} = e^{(-k)(-A_{ij})} = (e^{-k})^{-A_{ij}}$.

Let $\epsilon = e^{-k}$ $(\text{so that } k\rightarrow\infty\implies\epsilon\rightarrow0)$, and $B = -A$. Then $(\mathcal{A}_\epsilon)_{ij} = 1\epsilon^{B_{ij}} = (A_k)_{ij} $, so we can analyze the asymptotics of $A_k$ using the techniques in \citep{PertTheory}.

Suppose $V_\epsilon = (w_1\epsilon^{v_1}, w_2\epsilon^{v_2}, \dots, w_n\epsilon^{v_n})$ is an eigenvector predicted by Theorem 6.1 of \citep{PertTheory}.

$$ \lim_{k\rightarrow\infty}\left[\frac{1}{k}\log(V_\epsilon)\right] = \lim_{k\rightarrow\infty}\left[\frac{1}{k} \left(\log(w_1(e^{-k})^{v_1}), \log(w_2(e^{-k})^{v_2}), \dots ,\log(w_n(e^{-k})^{v_n}) \right) \right]$$
$$ = \lim_{k\rightarrow\infty}\left[\left( -v_1 + \frac{\log(w_1)}{k}, -v_2 + \frac{\log(w_2)}{k}, \dots, -v_n + \frac{\log(w_n)}{k}  \right)\right]$$
$$ = (-v_1, -v_2, \dots, -v_n)$$

However, Theorem 6.1 does not necessarily predict the Perron Eigenvector, and for cases where any of the weights $w_i$ are zero, it fails to give any characterization of the eigenvector in question. While the set of candidate vectors proposed by the algorithm does offer a look into the asymptotic behavior of the eigenvectors of $A_k$, there are no guarantees that $\mathscr{L}(A_k)$ is one of them at all.

Even when one of the predicted vectors $(-v_1, -v_2, \dots, -v_n)$ lies in the eigenspace $\sigma_{(max+)}$, it is not guaranteed to be the Perron Eigenvector. As a counterexample, consider
$$A = \left[ \begin{matrix} 0 & -3 & -4\\ -1 & 0 & -2 \\ -1 & -1 & 0\end{matrix}\right]$$
Theorem 6.1 predicts
$$v = \left[ \begin{matrix} 1\epsilon^{-1} \\ -1\epsilon^{-1} \\ 1\epsilon^0 \end{matrix}\right] \mapsto \left[ \begin{matrix} 1\\1\\0 \end{matrix}\right] \mapsto \left[ \begin{matrix} 0\\0\\-1 \end{matrix}\right] \in \mathbb{TP}^2$$

However, while $(0,0,-1) \in \sigma_{(max+)}(A)$, \: $\mathscr{P}_\infty(A) = (0, 4, 3.5)$

\newpage

\section{Future Work}

Clearly there is more work to be done to fully characterize the Perron Eigenvector $\mathscr{P}_\infty(A)$ in the general case. It is clear that there is a relationship between the eigenspace $\sigma_{(max+)}(A)$ and the desired eigenvector, but the connection is not straightforward. The methods in \citep{PertTheory,AsymFrench} offer a potential solution, but have too many singular cases to be useful in applications in their current state. However, extensions of their algorithms to handle the remaining eigenvectors would not only solve the transfer matrix eigenvector problem, but would fully characterize the asymtotics of a wider class of matrices $A_\epsilon$.

Proofs of the conjectures in section 4 would likely shed more light on the relationship between $\mathscr{P}_\infty(A)$ and $\sigma_{(max+)}(A)$. If one can answer why the Perron Eigenvector's convergence is determined by the tropical eigenvectors of the base matrix, it is likely to lead to a characterization of the desired eigenvector itself. 

Such a characterization could potentially lead to a way to define the tropical eigenvalue/eigenvector of a tensor, see \citep{PFext}.

Python code used for analysis and visualization is available at \url{https://github.com/bkustar}.

\section*{Acknowledgements}

Special thanks to Ngoc Tran and Rutvik Choudhary.

\newpage

\nocite{AsymFrench}
\nocite{PFext}
\nocite{TropIntro}
\nocite{TropFirstSteps}
\bibliographystyle{abbrvnat}
\setcitestyle{authoryear,open={((},close={))}}
\bibliography{references}
\end{document}